\documentclass[12pt]{amsart}
\usepackage{amscd}
\usepackage{amssymb}
\usepackage{hyperref}

\newcommand{\ZZ}{\mathbb Z}

\newtheorem{lemma}{Lemma}[section]

\newtheorem{theorem}[lemma]{Theorem}

\newtheorem{definition}[lemma]{Definition}
\newtheorem{remark}[lemma]{Remark}

\advance\headheight1.15pt
\begin{document}
\title[Irreducible free divisors in $\mathbb{P}^2$]
{A family of irreducible free divisors in $\mathbb{P}^2$}

\thanks{$^\dag$ Supported by the INdAM-COFUND Marie Curie Fellowship, Italy.}
\thanks{Secondary: 14C20, 14H51, 14B05.}
\author[Ramakrishna Nanduri]{Ramakrishna Nanduri$^\dag$}
\address{Dipartimento di Matematica, Universit\'{a} Degli Studi di Genova, Via Dodecaneso 35, I-16146, Italia.}
\email{nandurirk@gmail.com}

\begin{abstract}
An infinite family of irreducible homogeneous free divisors in 
$K[x,y,z]$ is constructed. Indeed, we identify sets of monomials 
$X$ such that the general polynomial supported on $X$ is a free 
divisor. 
\end{abstract}

\maketitle
\section{Introduction}
 The goal of this note is to construct a 
 family of irreducible homogeneous free divisors in 
 $K[x,y,z]$, where $K$ is any field. 

 The notion of free divisors appeared in investigations of discriminants of versal unfoldings of 
 isolated hypersurface singularities and they were formally defined and investigated by Saito, \cite{si80}. 
 The notion of a free divisor has played a fundamental role for understanding the structure of nonisolated 
 singularities such as discriminants, hyperplane arrangements etc, see \cite{br85}, \cite{beg09}, 
 \cite{da98}, \cite{da03}, \cite{da06}, \cite{ot92}, \cite{si80}, \cite{sc12}, \cite{te80}, \cite{te81a}, 
 \cite{te81b}, \cite{te83}, \cite{vs95}, \cite{y05}, to cite a few. 
 Many authors have established a number of natural constructions which yield free divisors, 
 for example see \cite{br85}, \cite{bc12}, \cite{da98}, \cite{dm91}, 
 \cite{gms11}, \cite{si80}, \cite{se09}, \cite{si05}, \cite{st09}, \cite{st12}, 
 \cite{te80}, \cite{te83}, \cite{t12} among others. 
 \vskip 2mm
 Here we give the definition of free divisors in the polynomial ring $K[x_1,\ldots, x_n]$. 
\begin{definition}
 A (formal) free divisor is a reduced polynomial 
 $F \in K[x_1,\ldots,x_n]$ such that its Jacobian ideal 
 $J(F) = (\frac{\partial F}{\partial x_1},\ldots,\frac{\partial F}{\partial x_n}, F)$
 is perfect of codimension $2$ in $K[x_1,\ldots,x_n]$. 
\end{definition}
If $F$ is a weighted homogeneous polynomial with respect to a weight 
 $w= (w_1,\ldots,w_n) \in \ZZ^n$ of positive $w$-degree, in $K[x_1,\ldots,x_n]$, then the Euler's 
equation, $w_1x_1 \frac{\partial F}{\partial x_1}+ \cdots+w_nx_n \frac{\partial F}{\partial x_n} = deg_w(F)F$, 
implies that $F \in (\frac{\partial F}{\partial x_1},\ldots, \frac{\partial F}{\partial x_n})$. 
Note that in $K[x,y]$, every reduced homogeneous polynomial is a free divisor. 
 The irreducible homogeneous free divisors are fascinating objects to study 
 and hard to find them. We list some of the known irreducible homogeneous free 
 divisors in the literature. In \cite{da02}, J. Damon gave a criterion for an 
 equisingular deformation of an isolated curve singularity of certain form, to 
 be a free divisor and by using this criterion, he showed that the Hessian 
 deformation defines a free divisor for nonsimple weighted homogeneous singularities. 
  In \cite{sms06}, A. Simis constructed a family of irreducible homogeneous 
 free divisors of degree $6$ in $K[x,y,z]$ called {\it Cayley sextics}, by 
 studying the depth of the Jacobian ideal, \cite[Proposition 4.4]{sms06}. 
 The discriminant polynomial of the generic binary form, is a homogeneous 
 free divisor, see \cite{ac07}. Granger, Mond and Schulze studied the linear 
 free divisors, which arising as discriminants, see \cite{gms11}. 
  Later Simis and Tohaneanu constructed a family of irreducible 
homogeneous free divisors in $K[x,y,z]$ of degree at least $5$, 
\cite[Proposition 2.2]{st12}. Also \cite{ns11}, \cite{st09}, \cite{t12}  
and \cite{si05} give some more insights on free divisors 
in three variables. Recently R.-O. Buchweitz and A. Conca 
studied the free divisors annihilated by many Euler vector fields and 
classification of binomial quasi-homogeneous free divisors, see \cite{bc12}. 
We study the following question. Describe sets of monomials 
$X$ in $K[x,y,z]$ of the same degree such that the sum of 
monomials in $X$ with general scalar coefficients, is a 
free divisor. 

 Motivated by Simis-Tohaneanu example, in this note we construct a 
family of irreducible free divisors in $\mathbb{P}^2$. For any 
$d\geq 5$, we show that 
$$x^{d-\alpha} F_1(x,y) + y^{\lfloor \frac{d}{2} \rfloor+\alpha+1} F_2(x,y) + x^{\beta}y^{d-\beta-1}z$$ 
  is an irreducible homogeneous free divisor in $K[x,y,z]$, 
  if $0 \leq \alpha, \beta \leq \lfloor\frac{d+1}{2}\rfloor-3$,  
  $0 \leq \alpha+\beta \leq \lfloor\frac{d+1}{2}\rfloor-3$ 
  and any homogeneous polynomials $F_1, F_2 \in K[x,y]$ 
  of degrees $\alpha$ and $(d-\lfloor \frac{d}{2} \rfloor-\alpha-1)$ 
  respectively such that $F_1$ is square-free, $x \nmid F_1F_2$ and $y \nmid F_1F_2$, 
see Theorem \ref{mainthm} below. This family of free 
divisors is larger than that of Simis-Tohaneanu's 
family. We prove our main theorem by describing the Saito matrix and the 
 structure of syzygies of the Jacobian ideal. 
\vskip 1mm
\noindent
{\bf Acknowledgement:} 
The author is grateful to Professor Aldo Conca for several useful 
discussions and suggestions on the content of this article. 
The author is thankful to Professor Aron Simis for his comments 
on the proof of the main theorem.  Also the author likes to sincerely 
thank INdAM-Cofunded by Marie Curie Actions, for their 
fellowship and the Department of Mathematics, University 
of Genova for their hospitality. 
\section{Main Results}
 For any real number $r$, let $\lfloor r \rfloor$ denote the 
 largest integer less than or equal to $r$ and $\lceil r \rceil$ 
 denote the smallest integer bigger than or equal to $r$. 
 For any polynomial $F \in K[x_1,\ldots,x_n]$, the vector 
$\nabla F = (\frac{\partial F}{\partial x_1},\ldots,\frac{\partial F}{\partial x_n})$, 
 is called the {\it gradient} of $F$. 
  We state a theorem of 
 K. Saito on free divisors, for the polynomial algebra $K[x_1,\ldots,x_n]$: 
\begin{theorem}[Saito's criterion, \cite{si80}] \label{saito80}
 A reduced polynomial $F \in K[x_1,\ldots,x_n]$ is a free divisor 
 if and only if there exists a $n\times n$ 
 matrix $A$ with entries in $K[x_1,\ldots,x_n]$ such that 
$\det(A) = F$ and $(\nabla F) A \equiv 0$ mod $(F)$.
\end{theorem} 
 The matrix appearing in the Saito's criterion, is called 
 a {\it discriminant matrix} or {\it Saito matrix} of the 
 polynomial. \\
 Now we prove a lemma which is useful in the proof of the 
 main theorem. 
 \begin{lemma} \label{lem1}
Let $d \geq 1$ be an integer and $v =\lfloor \frac{d}{2} \rfloor$. Let $\mu \in K$. 
Let $\alpha, \beta \geq 0$ be integers such that 
$\alpha + \beta \leq v$. 
Suppose $F_1, F_2 \in K[x,y]$ are homogeneous polynomials of degrees 
$\alpha, v-\alpha$ respectively. Assume $F_1$ is square-free, $x \nmid F_1F_2$ and $y \nmid F_1F_2$. Then there exist  
homogeneous polynomials $H_1, H_3, H_5 \in K[x,y]$ of 
degree $v$ satisfying \\ 
\begin{equation} \label{eq1}
\begin{bmatrix}
 -G_2 &  xG_1  & 0 \\~\\
 y\frac{\partial F_2}{\partial x}   &      ( y\frac{\partial F_2}{\partial y}+(d-v+\alpha)F_2 )    &  x^{\beta}y^{v-\alpha-\beta}\\
\end{bmatrix}  
\begin{bmatrix}
H_1 \\
H_3 \\
H_5
\end{bmatrix} 
=
\begin{bmatrix}
 \mu y^{v+\alpha} \\
 -\mu x^{2v-\alpha} 
\end{bmatrix}
\end{equation}
where $G_1 = \frac{\partial F_1}{\partial y}$ and $G_2 = -\left(x\frac{\partial F_1}{\partial x}+(d-\alpha)F_1\right)$.  
\end{lemma} 
\begin{proof}
 Let $S=K[x,y]$ and $N$ be the graded $S$-submodule of $S(-(v-\alpha))\oplus S(-\alpha)$ 
 generated by the columns of the coefficient matrix in \eqref{eq1}. 
 That is, $N$ is generated by 
 $\begin{bmatrix}
 -G_2  \\
 y\frac{\partial F_2}{\partial x} 
\end{bmatrix}$, 
$\begin{bmatrix}
   xG_1 \\
  y\frac{\partial F_2}{\partial y}+(d-v+\alpha)F_2    
\end{bmatrix}$ \mbox{ and }
$\begin{bmatrix}
  0 \\
  x^{\beta}y^{v-\alpha-\beta} 
\end{bmatrix} 
$. Let 
$\begin{bmatrix}
 t_1  \\
 t_2 \\
 t_3
\end{bmatrix}$ be a syzygy of $N$. Then we have 
 $t_1$ $\begin{bmatrix}
 -G_2  \\
 y\frac{\partial F_2}{\partial x} 
\end{bmatrix}$ $+ t_2$ 
$\begin{bmatrix}
   xG_1 \\
  y\frac{\partial F_2}{\partial y}+(d-v+\alpha)F_2    
\end{bmatrix}$ $+ t_3$
$\begin{bmatrix}
  0 \\
  x^{\beta}y^{v-\alpha-\beta} 
\end{bmatrix} = 0$. This implies that 
$-G_2t_1+xG_1t_2=0$ and 
$t_1y\frac{\partial F_2}{\partial x}+t_2\left(y\frac{\partial F_2}{\partial y}+(d-v+\alpha)F_2\right)+t_3x^{\beta}y^{v-\alpha-\beta}=0.$ 
Now we show that $xG_1, G_2$ is a regular sequence. It is enough to show that ht$(xG_1,G_2) = 2$. 
For $P \subset K[x,y]$ any prime ideal such that $(xG_1,G_2) \subseteq P$, 
we show that $P=(x,y)$. Since $(xG_1,G_2) \subseteq P$, we have  
$(x,G_2) \subseteq P$ or $(G_1,G_2) \subseteq P$. By using the assumption that $x \nmid F_1$, we have 
$(x,G_2)=(x,y^{\alpha})$, is a complete intersection of height $2$. Thus if $(x,G_2) \subseteq P$, then 
$P =(x,y)$, as required. Now assume that $(G_1,G_2) \subseteq P$. By using Euler's equation for $F_1$ we have 
$(G_1,G_2) = (\frac{\partial F_1}{\partial y}, x\frac{\partial F_1}{\partial x})$. Then we have 
 $(\frac{\partial F_1}{\partial y},x) \subseteq P$ or 
$(\frac{\partial F_1}{\partial y}, \frac{\partial F_1}{\partial x}) \subseteq P$. 
By using the assumption that $x \nmid F_1$, we have $(x, \frac{\partial F_1}{\partial y})= (x,y^{\alpha})$, is a complete 
intersection of height $2$. Thus if $(\frac{\partial F_1}{\partial y},x) \subseteq P$, then $P=(x,y)$. 
Now by the assumption that $F_1$ is square-free, we have $\frac{\partial F_1}{\partial x}, \frac{\partial F_1}{\partial y}$ 
is a regular sequence of height $2$. Thus if $(\frac{\partial F_1}{\partial y}, \frac{\partial F_1}{\partial x}) \subseteq P$, 
then $P=(x,y)$, as required. Thus in any case, $P=(x,y)$. 
Therefore $xG_1, G_2$ is a regular sequence. Then $-G_2t_1+xG_1t_2=0$, implies that 
$t_1=xG_1 \omega, t_2=G_2\omega$, for some $\omega \in K[x,y]$. 
Now $xy\omega G_1\frac{\partial F_2}{\partial x}
+\omega G_2 \left(y\frac{\partial F_2}{\partial y}+(d-v+\alpha)F_2\right)+t_3x^{\beta}y^{v-\alpha-\beta}=0$
implies that 
$$\omega \left[xyG_1\frac{\partial F_2}{\partial x}
+G_2 \left(y\frac{\partial F_2}{\partial y}+(d-v+\alpha)F_2\right)\right]+t_3x^{\beta}y^{v-\alpha-\beta}=0 $$
Now consider 
\begin{eqnarray*}
\left[xyG_1\frac{\partial F_2}{\partial x}
+G_2 \left(y\frac{\partial F_2}{\partial y}+(d-v+\alpha)F_2\right)\right] 
&=& xy\frac{\partial F_1}{\partial y}\frac{\partial F_2}{\partial x}-xy\frac{\partial F_1}{\partial x}\frac{\partial F_2}{\partial y} \\
&& -(d-\alpha)yF_1\frac{\partial F_2}{\partial y}-(d-v+\alpha)x\frac{\partial F_1}{\partial x}F_2 \\
&&  -(d-\alpha)(d-v+\alpha)F_1F_2.
\end{eqnarray*}
By the assumptions that $x\nmid F_1F_2$ and $y \nmid F_1F_2$, we have both $x$ and $y$ not divide 
$xyG_1\frac{\partial F_2}{\partial x}+G_2 \left(y\frac{\partial F_2}{\partial y}+(d-v+\alpha)F_2\right)$. Therefore 
$xyG_1\frac{\partial F_2}{\partial x}+G_2 \left(y\frac{\partial F_2}{\partial y}+(d-v+\alpha)F_2\right), x^{\beta}y^{v-\alpha-\beta}$ 
is a regular sequence. This implies that $t_3= -\omega_1\left[xyG_1\frac{\partial F_2}{\partial x}+G_2 \left(y\frac{\partial F_2}{\partial y}
+(d-v+\alpha)F_2\right) \right]$ and $\omega = x^{\beta}y^{v-\alpha-\beta} \omega_1$ for some $\omega_1 \in K[x,y]$.
Thus
$$\begin{bmatrix}
 t_1  \\
 t_2 \\
 t_3
\end{bmatrix} = \omega_1 
\begin{bmatrix}
 x^{\beta+1}y^{v-\alpha-\beta}G_1  \\
 x^{\beta}y^{v-\alpha-\beta} G_2 \\
 -xy\frac{\partial F_2}{\partial x}G_1 - G_2\left((d-v+\alpha)F_2+y \frac{\partial F_2}{\partial y} \right)
\end{bmatrix}
$$ for some $\omega_1 \in K[x,y]$. Therefore 
$\begin{bmatrix}
 \mbox{\tiny $x^{\beta+1}y^{v-\alpha-\beta}G_1$} \\
 \mbox{\tiny $x^{\beta}y^{v-\alpha-\beta} G_2$} \\
 \mbox{\tiny $-xy\frac{\partial F_2}{\partial x}G_1 - G_2\left((d-v+\alpha)F_2+y \frac{\partial F_2}{\partial y} \right)$}
\end{bmatrix}$
is the only generator of the syzygy module of $N$ and of degree $v$. 
Hence the minimal graded free resolution of $N$ is of the form: 
$$0 \longrightarrow S(-2v) \longrightarrow S^3(-v) \longrightarrow S(-(v-\alpha))\oplus S(-\alpha) \longrightarrow 0.$$
Where the first map is given by the matrix 
$\begin{bmatrix}
 \mbox{\tiny $x^{\beta+1}y^{v-\alpha-\beta}G_1$} \\
 \mbox{\tiny $x^{\beta}y^{v-\alpha-\beta} G_2$} \\
 \mbox{\tiny $-xy\frac{\partial F_2}{\partial x}G_1 - G_2\left((d-v+\alpha)F_2+y \frac{\partial F_2}{\partial y} \right)$}
\end{bmatrix}$. 
Therefore the Hilbert series of $S(-(v-\alpha))\oplus S(-\alpha)/N$ is  
$$\frac{z^{\alpha}+z^{v-\alpha}-3z^v+z^{2v}}{(1-z)^2}.$$
 This is equal to the polynomial 
 $$z^{\alpha}(1+z+\cdots+z^{v-\alpha-1})^2+z^{v-\alpha}(1+z+\cdots+z^{\alpha-1})(1+z+\cdots+z^{v-1}),$$
 of degree $2v-2$. This implies that 
$\left[ S(-(v-\alpha))\oplus S(-\alpha)/N \right]_i = 0$, for all $i \geq 2v-1$. Since 
the degree of $\begin{bmatrix}
 \mu y^{v+\alpha} \\
 -\mu x^{2v-\alpha} 
\end{bmatrix}$, is $2v$ in $S(-(v-\alpha))\oplus S(-\alpha)$, therefore 
the system \eqref{eq1} has a solution. 
\end{proof}
\noindent 
Now we prove the main theorem. 
\begin{theorem} \label{mainthm}
 Let $d \geq 5$ be an integer and $v =\lfloor \frac{d}{2} \rfloor$. 
 Then the homogeneous polynomial of degree $d$
  $$F=x^{d-\alpha} F_1(x,y) + y^{v+\alpha+1} F_2(x,y) + x^{\beta}y^{d-\beta-1}z$$ 
  is an irreducible  free divisor in $K[x,y,z]$ provided: 
  \begin{itemize}
 \item[(1)]  $\alpha, \beta \geq 0$ and $0 \leq \alpha+\beta \leq \lfloor\frac{d+1}{2}\rfloor-3$. 
 \item[(2)]  $F_1(x,y)$  is square-free, homogeneous of degree $\alpha$ and $x \nmid F_1, y \nmid F_1$.
 \item[(3)]  $F_2(x,y)$ is homogeneous of degree $d-v-\alpha-1$ and $x \nmid F_2, y \nmid F_2$.
\end{itemize}  
\end{theorem}
In the proof of the Theorem we will show also how to construct 
the Saito matrix of $F$. For instance, for  $\beta \neq 0$ we will 
show that the Saito matrix of $F$ has the following shape: 
  \[
      \frac{1}{d\mu }   
            \left[ {\begin{array}{lll}
             x & \mbox{\tiny $x^{\beta+1}y^{\gamma +2}\frac{\partial F_1}{\partial y}$}             
             &  \mbox{\tiny $H_1+x^{\beta}y^{\gamma +1}z\frac{\partial F_1}{\partial y}E$}\\  \\
             y & \mbox{\tiny $-x^{\beta}y^{\gamma +2}\left(x\frac{\partial F_1}{\partial x}+(d-\alpha)F_1\right)$}     
             & \mbox{\tiny $H_3-x^{\beta-1}y^{\gamma +1}z\left(x\frac{\partial F_1}{\partial x}+(d-\alpha)F_1\right)E$} \\  \\
             z & \mbox{\tiny $G_3-x^{\beta}y^{\gamma +1}zW$} 
               & \mbox{\tiny $H_5+H_6z-x^{\beta-1}y^{\gamma}z^2EW$} \\
                \end{array} } \right],
        \]
 
\noindent 
where 
\begin{itemize}
\item[(i)] $\gamma = d-v-3-\alpha-\beta$,  
\item[(ii)] $H_1, H_3, H_5 \in K[x,y]$ are  homogeneous of degree $v$,  
$H_6 \in K[x,y]$ is homogeneous of degree $v-1$,  
$E \in K[x,y]$ is linear if $d$ is odd and quadratic if $d$ is even. 
The nature of these polynomials is  described in the proof. 
\item[(iii)]  $W = \beta y \frac{\partial F_1}{\partial y}-(d-\beta-1)\left(x\frac{\partial F_1}{\partial x}+(d-\alpha)F_1\right)$, \\
\end{itemize} 
and 
$$G_3 = \left\{
 \begin{array}{ll}
\begin{vmatrix}
1 & -(\frac{2d-v}{\alpha})y \frac{\partial F_1}{\partial y}  &  \frac{d(v+\alpha+1)}{\alpha(d-\alpha)}x\frac{\partial F_1}{\partial x}\\~\\
1 & \frac{d(2d-v)}{\alpha(d-\alpha)}x\frac{\partial F_1}{\partial x} & -\frac{d-\alpha}{\alpha}y\frac{\partial F_1}{\partial y} \\~\\
0 & x\frac{\partial F_2}{\partial x}                                 & y\frac{\partial F_2}{\partial y}
\end{vmatrix}, & \mbox{ if } \alpha \neq 0, \\ ~~ \\
-d \left( y\frac{\partial F_2}{\partial y}+(d-v)F_2 \right) F_1, & \mbox{ if } \alpha = 0. 
\end{array}
\right.
$$
\begin{proof}
  Let $F = x^{d-\alpha} F_1 + y^{v+\alpha+1} F_2 + x^{\beta}y^{d-\beta-1}z$, 
 where $F_1, F_2 \in K[x,y]$ are homogeneous polynomials of degrees 
 $\alpha, d-v-\alpha-1$ respectively such that $F_1$ is square-free, $x \nmid F_1F_2$ and $y \nmid F_1F_2$. 
 Since $F$ is linear in $z$ and by the assumptions on $F_1$ and $F_2$, forces 
 that $F$ is irreducible. We prove the theorem for the case $\beta \neq 0$ 
 and the proof is similar for the case $\beta =0$. Assume $\beta > 0$. We prove $F$ has a Saito 
 matrix  say $A$, as stated above. That is, we need to show that the product 
 $\begin{bmatrix}
 \frac{\partial F}{\partial x} & \frac{\partial F}{\partial y} &  \frac{\partial F}{\partial z}
\end{bmatrix} A \equiv 0$ mod $(F)$ and $\det(A) = F$. 
Since $F$ is homogeneous, by Euler's equation we have 
$x \frac{\partial F}{\partial x}+y \frac{\partial F}{\partial y}
 +z \frac{\partial F}{\partial z} = d~F$. Therefore the first 
 entry of the above product matrix is congruent to zero modulo $(F)$. 
 Now it suffices to show that the last two entries of the above 
 product matrix are congruent to zero modulo $(F)$ and $\det(A) = F$. 
 In fact, we show that the second column of $A$ is a  
 syzygy of $J(F)$ of degree $v$, and there exist homogeneous polynomials 
 $H_1,H_3, H_5 \in K[x,y]$ of degree $v$, $H_6 \in K[x,y]$ 
 of degree $v-1$, such that the third column of $A$ is a 
 syzygy of $J(F)$ of degree $d-v-1$. Then by Saito's criterion, Theorem \ref{saito80}, 
 $F$ is a free divisor. We have  
 \begin{eqnarray*}
  \frac{\partial F}{\partial x} &=& x^{d-\alpha-1} \left( x \frac{\partial F_1}{\partial x}+(d-\alpha)F_1 \right) 
                                    + y^{v+\alpha+1}\frac{\partial F_2}{\partial x}+ \beta x^{\beta-1}y^{d-\beta-1}z, \\
  \frac{\partial F}{\partial y} &=& x^{d-\alpha} \frac{\partial F_1}{\partial y} + y^{v+\alpha} 
                                   \left( y\frac{\partial F_2}{\partial y}+(v+\alpha+1)F_2 \right)
                                    +(d-\beta-1)x^{\beta}y^{d-\beta-2}z, \\
  \frac{\partial F}{\partial z} &=&  x^{\beta}y^{d-\beta-1}.                                 
 \end{eqnarray*}
  Now we show that the second column of $A$ is a syzygy of $J(F)$. 
  That is, we show that  
 \begin{equation} \label{eq3}
 x^2y^2M G_1 \frac{\partial F}{\partial x} - xy^2MG_2\frac{\partial F}{\partial y} 
 + \left(G_3-xyzM(\beta yG_1+(d-\beta-1)G_2)\right)\frac{\partial F}{\partial z}= 0
 \end{equation}
 where $M = x^{\beta-1}y^{\gamma}, G_1 = \frac{\partial F_1}{\partial y}$ and 
  $G_2 = -\left(x\frac{\partial F_1}{\partial x}+(d-\alpha)F_1\right)$. Note that 
  $$G_3 = -\left( xy\frac{\partial F_2}{\partial x}G_1+\left(y\frac{\partial F_2}{\partial y}+(v+\alpha+1)F_2 \right)G_2 \right).$$
 Consider 
 \begin{eqnarray*}
  x^2y^2M G_1 \frac{\partial F}{\partial x} - xy^2MG_2  
 \frac{\partial F}{\partial y} + \left(G_3-xyzM(\beta yG_1+(d-\beta-1)G_2)\right) \frac{\partial F}{\partial z}
 \end{eqnarray*}
 \begin{eqnarray*}
 &=& x^2y^2M \frac{\partial F_1}{\partial y} \left( (d-\alpha)x^{d-\alpha-1}F_1 
 + x^{d-\alpha} \frac{\partial F_1}{\partial x} + y^{v+\alpha+1}\frac{\partial F_2}{\partial x} + \beta x^{\beta-1}y^{d-\beta-1}z \right) \\
& & + xy^2M G_2 \left( x^{d-\alpha} \frac{\partial F_1}{\partial y} + (v+\alpha+1) y^{v+\alpha} F_2 \right) \\
& & + xy^2M G_2 \left( y^{v+\alpha+1}\frac{\partial F_2}{\partial y} + (d-\beta-1) x^{\beta} y^{d-\beta-2}z \right) \\
& & - xy^{v+\alpha+2}M \left( xy\frac{\partial F_2}{\partial x} 
     \frac{\partial F_1}{\partial y} + G_2 (y\frac{\partial F_2}{\partial y}+(v+\alpha+1)F_2) \right) \\ 
 & & - x^2y^{v+\alpha+3}zM^2 \left(\beta y \frac{\partial F_1}{\partial y}+(d-\beta-1) G_2 \right)  \\
 &=& x^2y^2M \frac{\partial F_1}{\partial y} \left( (d-\alpha)x^{d-\alpha-1}F_1 + x^{d-\alpha} \frac{\partial F_1}{\partial x} 
 + y^{v+\alpha+1}\frac{\partial F_2}{\partial x} \right) \\
 & & + xy^2M G_2\left( x^{d-\alpha} \frac{\partial F_1}{\partial y} + (v+\alpha+1) y^{v+\alpha} F_2 
 + y^{v+\alpha+1} \frac{\partial F_2}{\partial y} \right) \\ 
 & & - xy^{v+\alpha+2}M 
 \left( xy\frac{\partial F_2}{\partial x} \frac{\partial F_1}{\partial y} + G_2 (y\frac{\partial F_2}{\partial y}+(v+\alpha+1)F_2) \right) \\
& & + z \left( \beta x^{\beta+1}y^{d-\beta+1} M \frac{\partial F_1}{\partial y}+ (d-\beta-1) x^{\beta+1}y^{d-\beta} M G_2 \right) \\ 
& & - x^{\beta+1}y^{d-\beta}z M \left( \beta y \frac{\partial F_1}{\partial y} + (d-\beta-1) G_2\right) \\
&=& 0 + z ~ 0 ~~~~~~~~~~~~\mbox{(by substituting $G_2$ value)} \\
&=& 0.
\end{eqnarray*}
 Thus \eqref{eq3} is a syzygy of $J(F)$. 
Now we show that the third column of $A$ is a syzygy of $J(F)$. 
We prove this for the case $d = 2v+1$, is odd. Assume $E = ax+by$. We show that  
{\small 
\begin{equation} \label{eq4}
  \left(H_1+xyzMH_2\right) \frac{\partial F}{\partial x}+\left(H_3 +yzMH_4\right) 
  \frac{\partial F}{\partial y} + \left(H_5+H_6z-Mz^2\left(\beta y H_2+(d-\beta-1)H_4\right)\right)\frac{\partial F}{\partial z} = 0
 \end{equation} }
 where $H_2 = G_1 (ax+by), H_4 = G_2 (ax+by)$ and $H_1, H_3, H_5 \in K[x,y]$ 
 are homogeneous polynomials of degree $v$ satisfying the system of equations \eqref{eq1} in the Lemma \ref{lem1} 
and 
\begin{eqnarray*}
H_6 &=& -\beta V_1 - (d-\beta-1) U_1 - \frac{\partial F_2}{\partial x} \frac{\partial F_1}{\partial y}(ax+by) 
  +a \frac{\partial F_2}{\partial y}\left(x\frac{\partial F_1}{\partial x}+(d-\alpha)F_1\right) \\
    & & + b \frac{\partial F_1}{\partial x} \left(y\frac{\partial F_2}{\partial y}+(d-v+\alpha)F_2\right)+W_1+W_2, 
\end{eqnarray*}
 where 
 \begin{eqnarray*}
 H_1  &=& xV_1+ * y^v, \\
 H_3  &=& yU_1+* x^v, \\ 
 a (d-v+\alpha) F_2 ( x\frac{\partial F_1}{\partial x} + (d-\alpha)F_1 )  &=& yW_1+ * x^v, \\ 
 b(d-\alpha)F_1 ( y \frac{\partial F_2}{\partial y} + (d-v+\alpha) F_2 )  &=& x W_2+* y^v, 
 \end{eqnarray*}
 where $*$ are elements of $K$. 
  From the expression of $H_6$, we have 
  \begin{equation} \label{eq2}
  \beta y H_1 + xy \frac{\partial F_2}{\partial x} H_2 
 + (d-\beta-1) x H_3 +\left(y\frac{\partial F_2}{\partial y}+(d-v+\alpha)F_2\right)H_4 + xy H_6 = 0.
 \end{equation}
  Now by comparing the pure powers of $x$ and $y$ in system \eqref{eq1} in the Lemma \ref{lem1} and equation \eqref{eq2}, 
  we get 
\begin{eqnarray*}
a &=& \frac{- \mu (d-\beta-1)}{(d-v+\alpha)^2 [F_2| x^{v-\alpha}]^2[\left(x\frac{\partial F_1}{\partial x}+(d-\alpha)F_1\right)| x^{\alpha}] },  \\
\mu &=& \frac{b (d-\alpha)^2}{\beta }[F_1|y^{\alpha}][(y\frac{\partial F_2}{\partial y}+(d-v+\alpha)F_2)|y^{v-\alpha}], 
\end{eqnarray*}  
  where $[H | x^s]$ denote the coefficient of $x^s$ of $H$, for any polynomial $H$ etc.
 Now consider, 
 {\small 
 \begin{eqnarray*}
 \left(H_1+xyzMH_2\right) \frac{\partial F}{\partial x} + \left(H_3 
 +yzMH_4\right) \frac{\partial F}{\partial y} 
 + (H_5+H_6z-Mz^2\left(\beta y H_2+(d-\beta-1)H_4\right)) \frac{\partial F}{\partial z}
 \end{eqnarray*} }
 it is clear that its coefficient of $z^2$ is zero. Now its 
 coefficient of $z = y^{v+\alpha+1}M ( \beta y H_1 + xy \frac{\partial F_2}{\partial x} \frac{\partial F_1}{\partial y}(ax+by)
 + (d-\beta-1) x H_3 -(y\frac{\partial F_2}{\partial y}+(d-v+\alpha)F_2)(x\frac{\partial F_1}{\partial x}+(d-\alpha)F_1)(ax+by)
 + xy H_6) + x^{d-\alpha}yM (-G_2H_2+G_1H_4)$. This is zero by equation \eqref{eq2} and the fact 
 $G_1H_4 = G_2H_2$. Now the coefficient of $z^0$ in the expression that we considered is equal to, \\ 
 {\small 
 $x^{d-\alpha-1} \left( (x\frac{\partial F_1}{\partial x}+(d-\alpha)F_1) H_1 + x \frac{\partial F_1}{\partial y} H_3 \right)$ \\  
 $+y^{d-v+\alpha-1} \left( y\frac{\partial F_2}{\partial x} H_1 + ((d-v+\alpha) F_2 + y \frac{\partial F_2}{\partial y}) H_3 
 + x^{\beta} y^{v-\alpha-\beta } H_5 \right)$ \\
 $= x^{d-\alpha-1} \left( -G_2 H_1 + x G_1 H_3 \right)  
 +y^{d-v+\alpha-1} \left( y\frac{\partial F_2}{\partial x} H_1 + ((d-v+\alpha) F_2 + y \frac{\partial F_2}{\partial y}) H_3 
 + x^{\beta} y^{v-\alpha-\beta } H_5 \right)$. }
 Since $x^{d-\alpha-1}, y^{d-v+\alpha-1}$ is a regular sequence, therefore 
 this expression is equal to zero if and only if the system \eqref{eq1} in the Lemma \ref{lem1} has 
 a solution. By the Lemma \ref{lem1}, the system \eqref{eq1} has a solution. 
\noindent 
Therefore \eqref{eq4} is a syzygy of $J(F)$. In a similar manner one can prove 
the even case also. Hence 
$\begin{bmatrix}
 \frac{\partial F}{\partial x} & \frac{\partial F}{\partial y} &  \frac{\partial F}{\partial z}
\end{bmatrix} A \equiv 0$ mod $(F)$.  
 If we know, the syzygies \eqref{eq3} and \eqref{eq4} are not 
 multiple of a syzygy, then by \cite[Lemma 1.1]{st12}, $J(F)$ is perfect. But 
by looking at the structure of the syzygies \eqref{eq3} and \eqref{eq4} it is 
difficult to say that one is not a multiple of another or both are not multiples 
of a syzygy.  
\vskip 1mm
Now we prove $\det(A) = F$. We have 
\begin{eqnarray*}
 d\mu \det(A) &=& x^2y^2MG_2 U -x \left( G_3-xyzM(\beta yG_1+(d-\beta-1)G_2) \right) \left(H_3+yzMH_4\right) \\
         && -x^2y^2M \frac{\partial F_1}{\partial y} \left[ yH_5+yzH_6-H_3z-yz^2M \left(\beta yH_2+(d-\beta)H_4 \right) \right] \\ 
         && +x^2y^3z^2M^2 \frac{\partial F_1}{\partial y} \left( \beta y\frac{\partial F_1}{\partial x}
            -(d-\beta)\left(x\frac{\partial F_1}{\partial x}+(d-\alpha)F_1\right) \right) \\
         && +\left(H_1+xyzMH_2\right)\left(yG_3-xy^2zM\left(\beta yG_1+(d-\beta)G_2 \right) \right).        
\end{eqnarray*}
Now, the coefficient of $z^2$ in $d\mu \det(A) =$
{\small 
\begin{eqnarray*}
    x^2y^2M^2 \left( \beta y G_1 + (d-\beta-1)G_2 \right) (H_4-yH_2) 
      + x^2y^2M\left( \beta y H_2 + (d-\beta-1)H_4 \right) (-G_2+yG_1).    
\end{eqnarray*} }
 This is equal to zero, because $H_2=G_1 E$, $H_4=G_2 E$ and $-G_2+yG_1= dF$. 
 Now the coefficient of $z$ in $d\mu \det(A) =$
 \begin{eqnarray*}
  && x^2y^2M G_2H_6 - xyM G_3H_4 +x^2yM H_3(\beta yG_1+(d-\beta-1)G_2)- x^2y^3M G_1H_6 \\ 
  && + x^2y^2MG_1H_3-xy^2MH_1(\beta yG_1+(d-\beta-1)G_2) - xy^2MG_2H_1 + xy^2M G_3H_2  
 \end{eqnarray*}
 Now substituting the system \eqref{eq1} of the Lemma \ref{lem1} in this, then this is equal to \\
 $xyM [ \mu (\beta+1)y^{v+\alpha+1}+(\beta+1-d)yG_2H_1 -G_3H_4 +yG_3H_2 -xy\frac{\partial F_2}{\partial x}G_2H_2$ \\ 
 $-G_2H_4( (d-v+\alpha)F_2+y \frac{\partial F_2}{\partial y} ) +xy^2 \frac{\partial F_2}{\partial x}G_1H_2 
 +(d-\beta-1)xy G_1H_3 +yG_1H_4( (d-v+\alpha)F_2+y \frac{\partial F_2}{\partial y} )  ] $. 
 Now by using the fact, $G_2H_2 = G_1H_4$ and system \eqref{eq1}, this expression is equal to 
 $xyM[ \mu dy^{v+\alpha+1}+xy\frac{\partial F_2}{\partial x}G_1H_4 
 - yG_1H_4((d-v+\alpha)F_2+y \frac{\partial F_2}{\partial y})+yG_1H_4 (-2x\frac{\partial F_2}{\partial x}+dF_2 )]$. 
 Now by using Euler's equation this is equal to $\mu d x^{\beta}y^{2v-\beta}$. Thus we showed that 
 the coefficient of $z$ in $d\mu \det(A)$ is equal to $\mu d x^{\beta}y^{2v-\beta}$. Now consider 
  \begin{eqnarray*}
 \mbox{the coefficient of $z^0$ in } d\mu \det(A) &=& x^2y^2M G_2H_5 -x G_3H_3 -x^2y^3M G_1H_5 +y H_1G_3 \\ 
                                             &=& -xG_3H_3+yG_3H_1 -dxF_1 [ \mu x^{2v-\alpha}-y\frac{\partial F_2}{\partial x}H_1 \\
                                             & & -H_3(y\frac{\partial F_2}{\partial y}+(d-v+\alpha)F_2)] \\
                                             & & \mbox{(by using the system \eqref{eq1} of the Lemma \ref{lem1})} \\
                                             &=& d\mu F_1 x^{2v-\alpha+1}+d xyF_2 G_1H_3-d yF_2G_2H_1 \\
                                             & & \mbox{(by substituting $G_3$ and Euler's equation )} \\                                   
                                             & & \mbox{(by using the system \eqref{eq1} of the Lemma \ref{lem1})}.
\end{eqnarray*}
Therefore 
\begin{eqnarray*}
\det(A) &=& \frac{1}{d \mu}\left[ d\mu F_1 x^{2v-\alpha+1}+d\mu F_2 y^{v+\alpha+1}+d\mu x^{\beta}y^{2v-\beta}z \right]\\
        &=& x^{2v-\alpha+1} F_1 + y^{v+\alpha+1} F_2 + x^{\beta}y^{2v-\beta}z \\
        &=& F.
 \end{eqnarray*}
 Thus $A$ is a discriminant matrix of $F$. 
 If $\beta = 0$, then the Saito matrix is same as $A$ except its last column. The last column 
 is the syzygy of $J(F)$ of the form: 
 $$\left(H_1-\lambda xy^{v-\alpha-1}z\frac{\partial F_1}{\partial y}\right) \frac{\partial F}{\partial x}+ 
 \left(H_3+\lambda y^{v-\alpha-1}z ((d-\alpha)F_1+x\frac{\partial F_1}{\partial x})\right)\frac{\partial F}{\partial y}$$  
 $$+\left(H_5+U z-\lambda (d-1)y^{v-\alpha-2}z^2((d-\alpha)F_1+x\frac{\partial F_1}{\partial x})\right)\frac{\partial F}{\partial z}=0,$$
 where $H_1, H_3, H_5 \in K[x,y]$ are as above, satisfying the system \eqref{eq1} of the Lemma \ref{lem1}, with $\beta=0$ and 
 $$U =((d-\alpha)F_1+x\frac{\partial F_1}{\partial x})\left(V_3- \frac{\lambda (d-v+\alpha)}{d-1} W_3 
 -\frac{\lambda}{d-1} x \frac{\partial F_1}{\partial y} \frac{\partial F_2}{\partial y} \right)
 $$
 $+ \frac{\lambda}{d-1} x^2 (\frac{\partial F_1}{\partial y})^2 \frac{\partial F_2}{\partial x} 
 + \mu y^{v+\alpha-1}$, where $\lambda \in K$, $x\frac{\partial F_1}{\partial y}F_2 = W_3y+*x^v$, 
 $H_1=yV_3+*x^v$, $*$ are elements of $K$. 
\end{proof}

\begin{remark}
Let $F$ be as in the Theorem \ref{mainthm}. The radical of the Jacobian ideal, $J(F)$, of $F$ is the ideal 
$(x,y)$. Hence $F$ has only one singular point $(0:0:1)$. The minimal graded free resolution 
of $J(F)$ is given by 
$$0 \longrightarrow S^2(-3v) \longrightarrow S^3(-2v) \longrightarrow S \longrightarrow 0$$
if $d=2v+1$, and 
$$0 \longrightarrow S(-(3v-2))\oplus S(-(3v-1)) \longrightarrow S^3(-(2v-1)) \longrightarrow S \longrightarrow 0$$ 
if $d= 2v$, where $S= K[x,y,z]$. From these, the multiplicity of the singular point $(0:0:1)$ of $F$ is equals to 
$$
 \left\{
\begin{array}{ll}
3v^2 & \mbox{ if } d=2v+1,  \\
3v^2-3v+1 & \mbox{ if } d=2v.
\end{array}
\right.
$$
\end{remark}
\begin{remark}
 Let $F_1,F_2$ and $F$ be as in the Theorem \ref{mainthm} except the assumption that $F_1$ is square-free. 
Then there is a computational evidence that $F$ is a free divisor with the last two columns of the 
Saito matrix of $F$ are syzygies of degrees need not be equal to $v$ or $v-1$. 
\end{remark}
\begin{remark}
 The monomial support of the free divisors in the Theorem \ref{mainthm} contains 
two intervals $\{x^{d},\ldots, x^{d-\alpha}y^{\alpha}\}$, 
$\{x^{\lfloor \frac{d}{2} \rfloor-\alpha}y^{d-\lfloor \frac{d}{2} \rfloor+\alpha}, 
 \ldots, y^d  \}$ and it is maximal. 
 There is a computational evidence that there are maximal monomial supports for irreducible 
free divisors in $\mathbb{P}^2$, for instance  
$$\{x^d\}\cup \{x^{d-v+2}y^{v-2}\}\cup \{x^{\lceil v/2 \rceil+2}y^{d-\lceil v/2 \rceil-2},
\ldots,x^5y^{d-5}\}\cup \{x^2y^{d-2}, xy^{d-1}, y^d\} \cup \{ y^{d-1}z \}$$  
and different from the family presented in the Theorem \ref{mainthm}, for $d=2v+1$. 
\end{remark}
  The results presented in this paper have been inspired and suggested by computations performed
 by computer algebra system CoCoA, \cite{cocoa}. 

\end{document}